\documentclass[dvipdfmx,4paper]{article}

\usepackage{amsmath,epsfig,amssymb,amsbsy,amscd,amsfonts,amstext,color,bm, amssymb,amsmath,amsthm,tikz,tikz-cd,epic,eso-pic}

\usepackage{graphicx}
\usepackage{enumerate}
\usepackage{epstopdf}
\usepackage[all]{xy}
\usepackage{mathtools}
\usepackage{hyperref}

\usepackage{geometry}
\allowdisplaybreaks[1]

\theoremstyle{definition}
\newtheorem{thm}{Theorem}
\newtheorem{defi}[thm]{Definition}
\newtheorem{prop}[thm]{Proposition}
\newtheorem{coro}[thm]{Corollary}
\newtheorem{lemma}[thm]{Lemma}
\newtheorem{exam}[thm]{Example}

\numberwithin{thm}{section}

\newcommand{\mb}{\mathbb}
\newcommand{\bC}{\mb{C}}
\newcommand{\bK}{\mb{K}}
\newcommand{\bL}{\mb{L}}
\newcommand{\bN}{\mb{N}}
\newcommand{\bR}{\mb{R}}
\newcommand{\bS}{\mb{S}}
\newcommand{\bT}{\mb{T}}
\newcommand{\bW}{\mb{W}}
\newcommand{\bZ}{\mb{Z}}

\newcommand{\cA}{\mathcal{A}}
\newcommand{\cB}{\mathcal{B}}
\newcommand{\cC}{\mathcal{C}}
\newcommand{\cL}{\mathcal{L}}
\newcommand{\cW}{\mathcal{W}}

\newcommand{\frS}{\mathfrak{S}}

\DeclareMathOperator{\Ker}{\mathrm{Ker}}
\DeclareMathOperator{\codim}{\mathrm{codim}}
\DeclareMathOperator{\rank}{\mathrm{rank}}

\newcommand{\Arr}{\mathbf{Arr}}
\newcommand{\gArr}{\Arr^\circ}

\newcommand{\abs}[1]{{\left\lvert#1\right\rvert}}

\title{Degeneration in discriminantal arrangements}
\author{Takuya Saito\thanks{
Affiliations: Institute for Chemical Reaction Design and Discovery, Hokkaido University
e-mail: saito@icredd.hokudai.ac.jp}}
\date{\today}

\begin{document}
\maketitle
\begin{abstract}
Discriminantal arrangements are hyperplane arrangements that are generalization of braid arrangements.
They are constructed from given hyperplane arrangements, but their combinatorics are not invariant under combinatorial equivalence.
However, it is known that the combinatorics of the discriminantal arrangements are constant on a Zariski open set of the space of hyperplane arrangements. 
In the present paper, we introduce $(\mathbb{T},r )$-singularity varieties in the space of hyperplane arrangements to classify discriminantal arrangements and show that the Zariski open set is the complement of $(\mathbb{T},r )$-singularity varieties.
We study their basic properties and operations and provide examples, including infinite families of $(\mathbb{T},r )$-singularity varieties.
In particular, the operation that we call degeneration is a powerful tool for constructing $(\mathbb{T},r )$-singularity varieties. As an application, we provide a list of $(\mathbb{T},r )$-singularity varieties for spaces of small line arrangements.
\end{abstract}

\renewcommand{\thefootnote}{}
\footnotetext{
\begin{flushleft}
\end{flushleft}
}\renewcommand{\thefootnote}{\arabic{footnote}}

\section{Introduction}

The discriminantal arrangement $\cB(\cA)$ consists of a collection of hyperplanes in the space of parallel translations of a given arrangement $\cA$.
Manin and Schechtman introduced this concept as a generalized braid arrangement \cite{MS89}.
The motivation for studying the discriminantal arrangement includes the study of higher Bruhat orders and higher braid groups.
Originally, the discriminantal arrangement was defined for a generic arrangement.
However, Bayer and Brandt extended the definition to arbitrary arrangements, where the hyperplanes in $\cB(\cA)$ correspond to circuits of $\cA$.
Applications of discriminantal arrangements include braided monoidal $2$-categories \cite{KaVo94}, vanishing cohomology of toric varieties \cite{Perl11}, and the maximum likelihood degree of very affine varieties of point configurations \cite{ABFKST23}.
For a general theory of hyperplane arrangements, refer to \cite{OTBook}.

Each point in the discriminantal arrangement $\cB(\cA)$ corresponds to a parallel translation $\cA^t$ of the original arrangement $\cA$.
Several invariants of parallel translations in a discriminantal arrangement have been studied in \cite{CFW21, FuWang23}.
From this perspective of parallel translations, a toric version, called the discriminantal toric arrangement, exists \cite{Paga19}.

The concepts equivalent to discriminantal arrangements have also been defined in various contexts.
Some concepts equivalent to this are a circuit basis for a binary matroid by Longyear \cite{Lo80}, concurrence geometry and the geometry of circuits by Crapo \cite{Cr84,Cr85}, the derived matroid by Oxley and Wang \cite{OxWa19}.
A similar concept has been pointed out as secondary polytopes, or more generally, fiber zonotopes \cite{BB97,EdRei96}.
All of these concepts are defined for representations of matroids.
Here, representations of matroids are essentially equivalent to hyperplane arrangements.
Some similar combinatorial constructions are adjoints \cite{Cheung74}, which generalize these to geometric lattices without representations, and combinatorial derived matroids, which provide a purely combinatorial definition in \cite{FrJuKu23}.
However, they are not matroid invariants except for combinatorial derived matroids.
In other words, they are constructions that cannot be uniquely determined from a given matroid.
Equivalently, in terms of hyperplane arrangements, the constant combinatorics of discriminantal arrangements cannot be obtained even if the arrangements have the same combinatorics.
In the context of discriminantal arrangements, Falk was the first to point out this fact in \cite{Fa94}.
Another fundamental example is given by Crapo \cite{Cr84} (see Example \ref{Exam:Crapo}).

Therefore, one of the problems in the theory of discriminantal arrangements is determining the possible discriminantal arrangements $\cB(\cA)$ for an arrangement $\cA$ with given combinatorics.
Let $\Arr(n,k)$ denote the set of all arrangements consisting of $n$ hyperplanes in $\bK^k$.
As is well known, this is equivalent to the quotient of the Grassmannian $\mathrm{Gr}_k(\bK^n)$ by the maximal torus $(\bK^\times)^n$.
We focus on the subset $\gArr(n,k)$, consisting of generic arrangements in $\Arr(n,k)$.
It is known that an arrangement on a Zariski open set in the realization space of a generic arrangement yields discriminantal arrangements with constant combinatorics \cite{Ath99}.
Arrangements on such a Zariski open set are called \textit{very generic}, and arrangements that are not are called \textit{non-very generic}.

A similar idea is also considered in coding theory.
Discriminantal arrangements are also known as arrangements that are spanned by points that are in a general position.
Falk first pointed out this fact in \cite{Fa94}.
From this perspective, the characteristic polynomial is calculated in \cite{KNT12, NuTa12}.
Brakensiek, Gopi, and Makam recently defined MDS$(\ell)$ codes based on the rank condition of the intersection of $\ell$ subspaces generated by the columns of the generator matrix of a code (\cite{BGM22}).
Just as an MDS code corresponds to a generic arrangement, for sufficiently large $\ell$, the MDS$(\ell)$ code corresponds to a very generic arrangement.
In the theory of higher order MDS codes, their existence depends on a coefficient field, and this problem is studied in \cite{BDG24}.
In this sense as well, determining a very/non-very generic arrangement is an important problem.
Furthermore, Jurrius and Pellikaan \cite{JuPe15} introduced the derived code, equivalent to the discriminantal arrangement, to determine extended coset weight enumerators.

Athanasiadis provided a combinatorial sufficient condition for a parallel translation to be non-very generic, but there are examples in \cite{SeYaArxiv22} where the converse does not hold.
Some other sufficient conditions for non-very genericity have been provided in \cite{SeYa22, SeYaArxiv22}. 
The set of non-very generic arrangements is an algebraic set; thus, it defines a stratification of $\Arr_\bK(n,k)$ (or the Grassmannian $\mathrm{Gr}_k(\bK^n)$), which is a refinement of matroid stratification (see \cite{Fa94}).
The Pappus hypersurfaces in $\mathrm{Gr}_3(\bK^n)$ were introduced in \cite{SaSeYa17} to study and classify Falk's examples.
It has been shown in \cite{SaSeYa19, SaSe24} that the intersections having the largest codimension of the Pappus hypersurfaces are related to the Hesse configuration and the dodecahedron, and the relationship with classical geometric objects is also interesting.
As a generalization of the Pappus hypersurface, we introduce the \textit{$(\bT,r )$-singularity variety} $V_{(\bT,r)}$ in $\Arr(n,k)$.
We show that this definition is appropriate in Theorem \ref{thm:NVGV2}; the set of non-very generic arrangements is the union of $(\bT,r )$-singularity varieties.
In order to examine the stratification of $\Arr(n,k)$, which is defined by $(\bT,r )$-singularity varieties, we must classify and determine them.
Some known simple descriptions of the stratification using algebraic sets in $\gArr(n,k)$ are found in \cite{SaSeYa19, SaSe24}.
However, explicit descriptions of the defining equations are generally unknown.
We construct $2n$-wheels as an infinite family of $(\bT,r )$-singularity varieties and provide their explicit equations in Theorem \ref{thm:wheel}.
Furthermore, we examine the behavior of $(\bT,r )$-singularity varieties for deleted/restricted arrangements (Theorem \ref{thm:degen}).
We show that the very generic arrangements are closed under the operations deletion and restriction.
In particular, deleting parallel elements, which we call \textit{degeneration}, is a powerful tool for finding non-very generic arrangements.
By applying this to $2n$-wheels, we obtain another infinite family of $(\bT,r )$-singularity varieties, which we call the $2n+2$-ladders (Proposition \ref{2n+2ladder}).
Determining all the $(\bT,r )$-singularity varieties for $8$ or fewer lines is possible (Theorem \ref{thm:8lines}).
Other studies on non-very generic arrangements are found in \cite{LiSe18, Yamagata23}.

The structure of this paper is as follows.
Section \ref{Sec:DefOfDisArr} includes the basic definitions of discriminantal arrangements and presents some examples of discriminantal arrangements with different combinatorics constructed from arrangements with the same combinatorics.
In Section \ref{Sec:NVArr}, we revisit the definition of non-very generic arrangements, translations, and intersections.
Furthermore, we introduce $(\bT,r )$-singularity varieties, which consist of non-very generic arrangements.
We discuss fundamental properties related to the necessary and sufficient conditions for non-very generic arrangements.
Furthermore, we give an inequality about codimension of $(\bT,r )$-singularity varieties.
In Section \ref{Sec:wheel}, we construct an infinite family of non-very generic line arrangements called $2n$-wheels and provide the defining equations of their $(\bT,r )$-singularity varieties.
Moreover, we give a sufficient condition for the existence of (degenerated) $2n$-wheels.
In Section \ref{Sec:Degen}, we review the discriminantal arrangement for multiarrangements.
Then, we discuss the deletion and restriction of arrangements, mainly focusing on the case of deleting parallel hyperplanes, which we refer to as degeneration.
In Section \ref{Sec:ExamNVVar}, by applying the construction from Section \ref{Sec:Degen} to $2n$-wheels, we construct the $2n+2$-ladders.
We determine all the defining equations for a minimal non-very generic line arrangement with eight lines.
Additionally, we discuss non-very generic line arrangements with nine lines.

\section{Discriminantal arrangements}\label{Sec:DefOfDisArr}
In this section, we provide the definitions of parallel translations of the arrangement, discriminantal arrangements, and examples.

\subsection{Notations and basics on hyperplane arrangements}
Throughout this paper, we suppose $\bK$ is a commutative field, and write $W$ a $k$-dimensional vector space over $\bK$.
We sometimes abbreviate the set $\{a, b, \dots, c\}$ as $ab \dots c$.

First, we recall the concepts of hyperplane arrangements.
Let $\cA = \{H_1,\ldots, H_n\}$ be a central hyperplane arrangement in $W$ with normal vectors $\alpha_i\in W^\ast$.
The \textit{combinatorics} (or the intersection lattice) of $\cA$ forms a graded lattice of the intersections of hyperplanes in $\cA$, ordered by reverse inclusion and denoted by $\cL(\cA)$.
The rank of $\cL(\cA)$ is defined as $\rank X = k - \dim X$.
A central arrangement $\cA$ is \textit{generic} if $\rank\left(\bigcap_{i\in S} H_i\right) = \min\{\abs{S}, k\}$ for each $S\subset [n] = \{1, \ldots, n\}$.
Two arrangements are \textit{combinatorially equivalent} if their combinatorics are lattice isomorphic and \textit{linearly equivalent} if a linear map exists between ambient spaces,  inducing a bijection between these arrangements.
For convenience, we require $\cA$ to be an \textit{essential} arrangement, that is, $\dim\bigcap_{H\in \cA}H=0$.
\textbf{Until Section \ref{Sec:wheel}, we assume that $\cA$ is a generic arrangement.}

\subsection{The space $\bS(\cA)$ of parallel translations of $\cA$ and its subspaces}
We define the space of parallel translations of an arrangement.
There is a natural isomorphism $\alpha^{-1}_i: \bK \to W/H_i; t_i \mapsto \alpha^{-1}_i(t_i)$, which induces the isomorphism from $\bK^n$ to the space $\bS(\cA) = \bigoplus_{i=1}^n W/H_i$, which is called \textit{the space of parallel translations} of $\cA$.
Define $\cA^t$ to be the image of $t = (t_i)_{i=1}^n \in \bK^n$.
Since $H_i^t = \alpha_i^{-1}(t_i)$ is a hyperplane in $W$, $\cA^t$ can be regarded as a hyperplane arrangement $\{H_i^t \mid i=1, \ldots, n\}$, which is a \textit{parallel translation} of $\cA$ by $t$.

For each subset $S$ of $[n]=\{1,\dots,n\}$, define the subspace $D_S = D_S(\cA)$ in $\bS(\cA)$ as the set of translations $\cA^t$ whose intersection of hyperplanes, indexed by $S$, is not empty, i.e., 
\begin{equation*}
D_S = \{\cA^t \in \bS(\cA) \mid \bigcap_{i\in S} H^t_i \neq \emptyset\}.
\end{equation*}
The dimension of $D_S$ is $\dim D_S=n-\abs{S}+\rank(\bigcap_{i\in S}H_i)=\min\{n,n+k-\abs{S}\}$.
Thus, $D_S$ is a hyperplane in $\bS(\cA)$ if and only if $\abs{S}=k+1$.
Let $\cC(\cA)$ denote $\{C\subset[n]\mid \abs{C}=k+1\}$.
Then $D_C$ forms a hyperplane in $\bS(\cA)$ with the normal $\alpha_C$, which is the Laplace expansion of the matrix $\begin{pmatrix}e^\ast_{i_1}&\dots&e^\ast_{i_{k+1}}\\ \alpha_{i_1}&\dots&\alpha_{i_{k+1}}\end{pmatrix}$ along the first row, for each $C=\{i_1,\dots,i_{k+1}\}\in \cC(\cA)$.
Here the $e_i$'s, in which each $e_i$ spans $W/H_i$ for each $i\in[n]$, are the canonical basis of $\bS(\cA)$, and the $e^\ast_i$ are the dual basis.

The \textit{discriminantal arrangement} associated to $\cA$ is the hyperplane arrangement
\begin{equation*}
\cB(\cA)=\{D_C\subset \bS(\cA)\mid C\in \cC(\cA)\}.
\end{equation*}

If the original arrangements are linearly equivalent, then the discriminantal arrangements associated to them are also linearly equivalent.
When $k=1$, $\cB(\cA)$ is the same as the braid arrangement, which consists of $D_{ij}=\{(t_1,\dots,t_n)\in\bK^n\mid t_i=t_j\}$ for all $1\leq i<j\leq n$.
The normal vectors of $\cA$ consist of non-zero scalars; thus, $\cC(\cA)=\binom{[n]}{2}$.
The hyperplanes are $D_{ij}=\{\cA^t\in \bS(\cA)\mid {t_i}\cap{t_j}\neq \emptyset\}=\{\cA^t\in \bS(\cA)\mid t_i=t_j\}$.

Let us define a representative of an intersection $X\in \cL(\cB(\cA))$.
If $\cA^{t_1},\cA^{t_2}$ are parallel translations of $\cA$ such that $\cL(\cA^{t_1})\cong \cL(\cA^{t_2})$, then we have that $\cA^{t_1}\in D_S$ if and only if $\cA^{t_2}\in D_S$ for any $S\in 2^{[n]}$.
We call an element $\cA^t$ of $X$ a \textit{representative} of $X$, and write $[\cA^t]$ as $X$ if and only if $\cA^t \in X\setminus(\bigcup_{Y\in\cL(\cB(\cA)),Y\subsetneq X}Y)$.
If  $\cA^{t_1},\cA^{t_2}$ are both representatives of  $X$, then $\cL(\cA^{t_1})\cong \cL(\cA^{t_2})$.
We call $D_S$ a \textit{component} of $X$ if $\abs{S}\geq k+1$, $X\subset D_S$ and $X\not\subset D_{S\cup\{i\}}$ for all $i\in[n]\setminus S$.
Define the \textit{canonical presentation} of $X=[\cA^t]$, denoted by $\bT(X)=\bT([\cA^t])$, as the set consisting of all the components of $X$, i.e.
\begin{equation*}
\bT(X)=\{S\subset [n]\mid D_S \mbox{ is a component of }X\}
\end{equation*}

The following example is fundamental.

\begin{exam}[Crapo \cite{Cr84}]\label{Exam:Crapo}
	Let $n=6$ and $k=2$, and let the coefficient field be the complex field $\bC$.
	Take $\lambda\in \bC\setminus\{0,\frac{1}{2},1,2\}$ and consider the vectors
	\begin{equation*}
	(\alpha_1\alpha_2\alpha_3\alpha_4\alpha_5\alpha_6)=\begin{pmatrix}1&2&1&1&0&\lambda\\0&1&1&2&1&1\end{pmatrix}.
	\end{equation*}
	Let $\cA$ be a generic arrangement with the normal vectors $\{\alpha_1 ,\alpha_2, \alpha_3, \alpha_4, \alpha_5, \alpha_6\}$ (Fig~\ref{fig;Ceva1}).
	We have $\cC(\cA)=\binom{[6]}{3}=\{S\subset [6]\mid \abs{S}=3\}$.
	Since $\cA$ is generic, $\cL(\cA)$ is independent of the choice of $\lambda$.
	We now consider the intersection $X$ of hyperplanes $D_{123},D_{156},D_{246},D_{345}$ in $\cB(\cA)$.
	We have that
	\begin{align*}
		\alpha_{123}=&e^\ast_1-e^\ast_2+e^\ast_3,&
		\alpha_{156}=&-\lambda e^\ast_1-e^\ast_5+e^\ast_6,\\
		\alpha_{246}=&(1-2\lambda)e^\ast_2+(\lambda-2)e^\ast_4+3e^\ast_6,&
		\alpha_{345}=&e^\ast_3-e^\ast_4+e^\ast_5.
	\end{align*}
	\begin{figure}[t]
	\begin{tabular}{ccc}
	\begin{minipage}{0.3\linewidth}
	\centering
	\begin{tikzpicture}[scale=2]
		\draw (0,-1)--(0,1) node[pos=-0.1]{$H_{1}$};
		\draw (-0.45,0.9)--(0.45,-0.9) node[pos=1.1]{$H_{2}$};
		\draw (-0.7,0.7)--(0.7,-0.7) node[pos=1.1]{$H_{3}$};
		\draw (-0.9,0.45)--(0.9,-0.45) node[pos=1.1]{$H_{4}$};
		\draw (-1,0)--(1,0) node[pos=1.1]{$H_{5}$};
		\draw (-0.9,-0.45)--(0.9,0.45) node[pos=1.1]{$H_{6}$};
	\end{tikzpicture}
	\caption{$\cA=\cA^0$ in Example \ref{Exam:Crapo}}\label{fig;Ceva1} 
	\end{minipage}&
	\begin{minipage}{0.3\linewidth}
	\centering
	\begin{tikzpicture}[scale=2.5]
		\draw (0,-0.2)--(0,1.2) node[pos=-0.1]{$H^t_{1}$};
		\draw (0.6,-0.2)--(-0.1,1.2) node[pos=-0.1]{$H^t_{2}$};
		\draw (-0.2,1.2)--(1.2,-0.2) node[pos=-0.1]{$H^t_{3}$};
		\draw (-0.2,0.6)--(1.2,-0.1)node[pos=-0.1]{$H^t_{4}$};
		\draw (-0.2,0)--(1.2,0) node[pos=-0.1]{$H^t_{5}$};
		\draw (-0.2,0.3)--(0.8,0.6) node[pos=1.1]{$H^t_{6}$};
	\end{tikzpicture}
	\caption{a representative of $D_{123}\cap D_{345}$}\label{fig:Ceva2} 
	\end{minipage}&
	\begin{minipage}{0.3\linewidth}
	\centering
	\begin{tikzpicture}[scale=2.5]
		\draw (0,-0.2)--(0,1.2) node[pos=-0.1]{$H^t_{1}$};
		\draw (0.6,-0.2)--(-0.1,1.2) node[pos=-0.1]{$H^t_{2}$};
		\draw (-0.2,1.2)--(1.2,-0.2) node[pos=-0.1]{$H^t_{3}$};
		\draw (-0.2,0.6)--(1.2,-0.1)node[pos=-0.1]{$H^t_{4}$};
		\draw (-0.2,0)--(1.2,0) node[pos=-0.1]{$H^t_{5}$};
		\draw[dashed] (-0.2,-0.2)--(0.8,0.8) node[pos=1.1]{$H^t_{6}$ $(\lambda=-1)$};
		\draw[dash dot] (-0.2,-0.1)node[at={(1.2,0.5)}]{$H^t_{6}$ $(\lambda\neq-1)$}--(0.8,0.4);
	\end{tikzpicture}
	\caption{a representative of $D_{123}\cap D_{156}\cap D_{345}$}\label{fig:Ceva3} 
	\end{minipage}
	\end{tabular}
	\end{figure}
	
	\noindent
	Then
	\begin{equation*}
	\rank X=6-\dim\langle\alpha_{123},\alpha_{156},\alpha_{246},\alpha_{345} \rangle=\left\{\begin{aligned}4&&\mbox{if }\lambda\neq-1,\\ 3&& \mbox{if }\lambda=-1.\end{aligned} \right.
	\end{equation*}
	Therefore, we get the representatives as in Figure \ref{fig:Ceva3} and the canonical presentations 
	\begin{equation*}
	\bT(X)=\left\{\begin{aligned}\{123456\}&&\mbox{if }\lambda\neq-1,\\ \{123,156,246,345\}&& \mbox{if }\lambda=-1.\end{aligned} \right.
	\end{equation*}
	A representative of $X$ is $\cA^{(0,0,0,0,0,0)}=\cA^0$ if $\lambda\neq-1$, or $\cA^{(0,1,1,1,0,0)}$ if $\lambda=-1$.
	
	Geometrically, one of the representatives of $D_{123}\cap D_{345}$ is $\cA^{(0,1,1,1,0,t_6)}$ for a generic $t_6$ (see Figure \ref{fig:Ceva2}).
	Then we obtain
	\begin{equation*}
		H^t_1\cap H^t_2\cap H^t_3=\{(0,1)\},\,H^t_3\cap H^t_4\cap H^t_5=\{(1,0)\}, \,H^t_1\cap H^t_5=\{(0,0)\}.
	\end{equation*}
	By translating $H^t_6$ through the intersection of $H^t_1$ and $H^t_5$, we find that $t_6=0$.
	This means that $\cA^{(0,1,1,1,0,0)}$ belongs to $D_{156}$.
	Now $H^t_2\cap H^t_4=\{(1/2,1/2)\}$, and $H^t_6$ passes through the intersection if and only if $\lambda=-1$.
	In other words, $D_{123}\cap D_{156}\cap D_{345}\subset D_{246}$ if $\lambda=-1$.
	If $\lambda\neq-1$, then we have $D_{123}\cap D_{156}\cap D_{345}\not\subset D_{246}$.
	Therefore, we can determine the rank of $X$.
	In some cases, it is possible to determine the rank of the intersection by observing the translation of $\cA$ in this way.
	
	More generally,over any commutative coefficient field, a discriminantal arrangement associated with a line arrangement $\cA$, consisting of $6$ lines, has an intersection $D_{123}\cap D_{156}\cap D_{246}\cap D_{345}$ of rank $3$ if and only if 
	\begin{equation}\label{Eq:quad}
	\Delta_{16}\Delta_{24}\Delta_{35}-\Delta_{15}\Delta_{26}\Delta_{34}=0
	\end{equation}
	where we use the symbol $\Delta_{ij}=\mathrm{det}(\alpha_i,\alpha_j)$.
	This equation was shown in \cite{Cr84, SaSe24}.
	We call a representative of the non-very generic intersection a Crapo configuration, whose canonical presentation is $\bT$ or its permutation $\sigma.\bT=\{\{\sigma.i_1,\sigma.i_2,\sigma.i_3\}\mid i_1i_2i_3\in\bT\}$ for $\sigma\in \frS_n$.
	Notice that the canonical presentation combinatorially determines this condition.
	This is an example of a discriminantal arrangement that is not invariant under the combinatorial equivalence relation.
\end{exam}

A phenomenon where intersections of hyperplanes in $\cB(\cA)$ with the same indices have different ranks also occurs for other values of $n$ and $k$.
The following example differs from the previous two examples in that the intersections of hyperplanes in $\cB(\cA)$ with the same canonical presentations have different ranks.
The next section explains the main difference between Example \ref{Exam:Crapo} and Example \ref{Exam:SeYa}.

\begin{exam}[Settepanella, Yamagata {\cite[Example 4.3.]{SeYaArxiv22}}]\label{Exam:SeYa}
	Let $n=10,k=3$, and let the coefficient field be the complex field $\bC$.
	Consider  $C_1 =\{1,2,3,4\}, C_2 =\{1,5,6,7\},\allowbreak C_3 =\{2,5,8,9\}, C_4 =\{3,6,8,10\}$ and $C_5 =\{4,7,9,10\}$.
	For a sufficiently generic arrangement $\cA$, we get $\rank \bigcap_{i=1}^5 D_{C_i}=5$.
	However, if we take normal vectors $\alpha_i$ of $H_i$ $(i=1,\dots,10)$ as follows;
	\begin{equation*}
	(\alpha_1,\ldots,\alpha_{10})=
	\begin{pmatrix}
		0&20&2&3&0&1&1&4&314&139\\
		10&0&-3&1&0&-1&2&-1&-40&30\\
		3&-9&0&0&1&1&2&-3&-197&-43
	\end{pmatrix},
	\end{equation*}
	then $\rank \bigcap_{i=1}^5 D_{C_i}=4$.
	The difference from the previous example is that the canonical presentation of intersection $X=\bigcap_{i=1}^5 D_{C_i}$ is the same $\{C_1,C_2,C_3,C_4,C_5\}$ in both cases of $\rank X=4$ or $5$.
	Hence, this condition does not depend solely on combinatorics.
\end{exam}


\section{Non-very generic arrangements}\label{Sec:NVArr}
As we saw in the previous section, discriminantal arrangements with different combinatorics can be obtained from arrangements with the same combinatorics.
We define classes of original arrangements based on this problem.
First, Bayer and Brandt defined very generic arrangements for the real case.
An arrangement $\cA$ consisting of $n$ hyperplanes in $\bR^k$ is \textit{very generic} if the cardinality of $\cL(\cB(\cA))$ is the largest possible for discriminantal arrangements associated to essential arrangements with $n$ hyperplanes in $\bR^k$.

If $\cA$ is generic, then for all translations $\cA^t$ of $\cA$, each $k$-element subset of $[n]$ is contained in at most one member of $\bT([\cA^t])$.
Hence, we have that if a family $\bT$ of subsets of $[n]$ is a canonical presentation of some intersection of a discriminantal arrangement, then $\bT$ satisfies the following conditions:
\begin{itemize}
\item[(Q0)] for each elements $S_1,S_2\in\bT$, $S_1=S_2$ if $S_1\subset S_2$,
\item[(Q1)] each element of $\bT$ has at least $k+1$ elements,
\item[(Q2)] each $k$-element subset of $[n]$ is contained in at most one member of $\bT$.
\end{itemize}
In order to consider all possible combinatorics of parallel translations of the generic arrangement, we define the set $Q(n,k)=Q([n],k)$ be the collection of families of subsets of $[n]$ whose members satisfy the conditions (Q0), (Q1), and (Q2).
The order on $Q(n,k)$ is defined as follows: for $\bT_1$ and $\bT_2$, $\bT_1\leq\bT_2$ if for every $S_1\in \bT_1$, there exists $S_2\in\bT_2$ such that $ S_1\subset S_2$.
\textbf{From now on, we always assume that a set $\bT$ satisfies (Q0).}

Following the works of Bayer--Brandt and Athanasiadis, we define a map $\nu:Q(n,k)\to\bN$ and a subposet $P(n,k)$ of $Q(n,k)$.
\begin{defi}[Bayer--Brandt \cite{BB97}, Athanasiadis \cite{Ath99}]\label{def:BBAcond}
Let $\bT$ be an element in $Q(n,k)$ and define 
\begin{equation*}
\nu(\bT)=\sum_{S\in\bT}(\abs{S}-k)
\end{equation*}
for each $\bT\in Q(n,k)$.
The family $\bT$ satisfies the \textit{Bayer-Brandt-Athanasiadis condition \textup{(}the BBA-condition\textup{)}} if 
\begin{itemize}
\item[\textup{(Q1)}] each element of $\bT$ has at least $k+1$ elements, and
\item[\textup{(P)}] $\abs{\bT'}>1$, then $\nu\left(\left\{\bigcup_{S\in\bT'}S\right\}\right)>\nu(\bT')$ for all $\bT' \subset\bT$.
\end{itemize}
The poset $P(n,k)$ is a subposet of the poset $Q(n,k)$ graded by $\nu$.
It consists of all $\bT\in Q(n,k)$ satisfying the BBA-condition.
\end{defi}

The following theorem was conjectured by Bayer--Brandt, and proved by Athanasiadis.
\begin{thm}[Bayer--Brandt \cite{BB97}, Athanasiadis \cite{Ath99}]
Let $n\geq k+1\geq 2$.
There exists a rank $k$ real arrangement $\cA$ consisting of $n$ hyperplanes, which is very generic and $\cL(\cB(\cA))$ isomorphic to $P(n,k)$.
The isomorphism is given by the canonical presentation $\bT:\cL(\cB(\cA))\to P(n,k)$.
\end{thm}
We define very generic arrangements over general fields.
\begin{defi}
Let $\cA$ be a generic arrangement with $n$ hyperplanes in $k$ dimensional space $W$, and let $X$ be an intersection in $\cL(\cB(\cA))$.
\begin{itemize}
\item an arrangement $\cA$ is called \textit{very generic} if $\cL(\cB(\cA))\cong P(n,k)$;
\item an intersection $X\in \cL(\cB(\cA))$ is \textit{very generic} if $\bT(X)\in P(n,k)$ and $\rank X=\nu(\bT(X))$;
\item a translation $\cA^t$ is \textit{very generic} if $\cA^t$ is representative of some very generic intersection.
\end{itemize}
Otherwise, they are called \textit{non-very generic}.
A non-very generic intersection $X$ is \textit{minimal} if $Y$ is very generic whenever $X\subsetneq Y\in\cL(\cB(\cA))$.
\end{defi}

Notice that (non-)very generic intersections were first defined in \cite{SaSe24}; however, due to the ambiguity of their original definition, we will adopt the above definition.
The condition of very generic intersections means that if $X=\bigcap_{S\in\bT(X)}D_S$ is very generic then $D_S$ intersect transversally where $S\in \bT(X)$, i.e. $\rank X=\sum_{S\in \bT(X)} \rank D_S$.
Furthermore, if $\bT(X)$ fails the BBA-condition, then $X$ is non-very generic.
However, the converse is not necessarily true.
Examples of such cases occur when $\bT(X)\in P(n,k)$ and $\rank X<\nu(\bT(X))$; one of the examples is in Example \ref{Exam:SeYa}.

\begin{prop}
Let $X\in \cL(\cB(\cA))$. Then $\rank X\leq \nu(\bT(X))$.
\end{prop}
\begin{proof}
By induction on the size of $\bT=\bT(X)$.
If $\abs{\bT}=1$, then we have $\rank X=\nu(\bT(X))$.
Assume that $\rank X'\leq \nu(\bT(X'))$ holds for every $X'$ such that $\abs{\bT(X')}=m-1$.
If $\bT=\{S_1,\dots,S_m\}$, then $X=\bigcap_{i=1}^{m-1}D_{S_i}\cap D_{S_m}$.
By the semimodularity of the rank function, 
\begin{align*}
\rank X\leq&\rank X+\rank\left(\left(\bigcap_{i=1}^{m-1}D_{S_i}\right)+ D_{S_m}\right) \\
\leq&\rank \bigcap_{i=1}^{m-1}D_{S_i}+\rank D_{S_m}\leq \nu\left(\bT\left(\bigcap_{i=1}^{m-1}D_{S_i} \right)\right)+\abs{S_m}-k=\nu(\bT(X)).\qedhere
\end{align*}
\end{proof}

In the rest of this section, we introduce the $(\bT,r )$-singularity varieties.

Let $\Arr(n,k)$ denote the set of linear isomorphism classes of essential central multiarrangements consisting of $n$ hyperplanes in $W=\bK^k$, and let $\gArr(n,k)\subset\Arr(n,k)$ denote the set of generic arrangements.
If $\cA$ is a generic arrangement, then their normal vectors $\alpha_1,\ldots,\alpha_n$ can be written as follows, by appropriate coordinate transformations and scalar products:
\begin{equation*}
\left(\alpha_1,\ldots,\alpha_n\right)=
\begin{pmatrix}
1&0&\cdots&0&0&1&a_{1,k+2}&\dots&a_{1,n}\\
0&1&&0&0&1&a_{2,k+2}&\dots&a_{2,n}\\
\vdots&&\ddots&&\vdots&\vdots&\vdots&&\vdots\\
0&0&&1&0&1&a_{k-1,k+2}&\dots&a_{k-1,n}\\
0&0&\cdots&0&1&1&1&\dots&1\\
\end{pmatrix}.
\end{equation*}
Hence, $\gArr(n,k)$ is isomorphic to a Zariski open set of a $(k-1)\times(n-k-1)$ dimensional affine space.
Note that all representatives of the isomorphism class of arrangements are written as $\left(\alpha_1,\ldots,\alpha_n\right)\in (W^\ast)^n$, and all lists of normal vectors that are transformed by the automorphism of $W$ and the scaling of each vector are included in the same equivalence class.
We have $\Arr(n,k)\subset \bK^{k\times n}/\mathrm{GL}(W)\times (\bK^\times)^n\cong\mathrm{Gr}_k(\bK^n)/(\bK^\times)^n$.
Therefore, we can treat $\Arr(n,k)$ as a quotient of the Grassmannian $\mathrm{Gr}_k(\bK^n)$ and will use Pl\"{u}cker coordinates for $\Arr(n,k)$.

\begin{defi}
Let $\bT$ be an element in $Q(n,k)$, and let $r$ be a non-negative integer.
Define
\begin{align*}
V_{(\bT,r)}&={\left\{\cA\in \Arr(n,k)\left| \rank \bigcap_{S\in \bT}D_S(\cA)\leq r\right\}\right.},\\
V^\circ_{(\bT,r)}&={V}_{(\bT,r)}\cap\gArr(n,k).
\end{align*}
We call $V_{(\bT,r)}$ the \textit{$(\bT,r )$-singularity variety} of $(\bT,r)$ if $V_{(\bT,r)}\neq \emptyset$ and $V_{(\bT,r)}\neq \Arr(n,k)$.
\end{defi}
Note that $V_{(\bT,r)}$ is defined by the vanishing of all the $r$-minors of the matrix $(\alpha_C)_{C\subset S\in \bT}$, so $V_{(\bT,r)}$ is an algebraic set in $\Arr(n,k)$. However, $V_{(\bT,r)}$ is not always irreducible.

We know that $\cA\in V_{(\bT_0,r)}$ if $\bT([\cA^t])=\bT_0$ for some $t\in\bK^n$.
Therefore, if a coefficient field is the real or complex numbers and $\bT\in P(n,k)$, then we have $V_{(\bT,\nu(\bT))}=\Arr(n,k)$.
Moreover, the following proposition holds.
\begin{prop}\label{prop:NVGV}
Let $\bT,\bT'\in Q(n,k)$ and let $r$ and $r'$ be non-negative integers.
If $r\leq r'$ and $\bT\geq\bT'$, then $V_{(\bT,r)}\subset V_{(\bT',r')}$.
\end{prop}
\begin{proof}
If $\bT\geq\bT'$, then it holds that $\bigcap_{S\in \bT}D_S\subset \bigcap_{S\in \bT'}D_S$ and so $\rank \bigcap_{S\in \bT}D_S\geq \rank \bigcap_{S\in \bT'}D_S$.
Hence, we have
\begin{equation*}
{\left\{\cA\in \Arr(n,k)\left| \rank \bigcap_{S\in \bT}D_S(\cA)\leq r\right\}\right.}
\subset{\left\{\cA\in \Arr(n,k)\left| \rank \bigcap_{S\in \bT'}D_S(\cA)\leq r'\right\}\right.}.\qedhere
\end{equation*}
\end{proof}

There exists a rank decreasing surjection $P(n,k)\to\cL(\cB(\cA));\bT\mapsto\bigcap_{S\in\bT}D_S$.
Hence, we have $\nu(\bT)\geq \rank \bigcap_{S\in\bT}D_S$ for $\bT\in P(n,k)$.
Therefore, for an intersection $X$ in $\cL(\cB(\cA))$, 
\begin{equation*}
\rank X<\min\left\{\nu(\bT')\left|\bT'\in P(n,k), \bT(X)<\bT'\right\}\right.
\end{equation*}
holds.
In particular, if $X$ is very generic, then $\min\left\{\nu(\bT')\left| \bT'\in P(n,k), \bT(X)<\bT'\right\}\right. =\rank X+1$.
We have $V^\circ_{(\bT,r)}=\gArr(n,k)$ if $r\geq \min\left\{\nu(\bT')\left|\bT'\in P(n,k), \bT<\bT'\right\}\right.$.

For simplicity, we write $V_{\bT}$ and $V^\circ_{\bT}$ as $V_{(\bT,r)}$ and $V^\circ_{(\bT,r)}$ if $\min\left\{\nu(\bT')\left| \bT'\in P(n,k), \bT<\bT'\right\}\right.  =r+1$.
Then we have $V_\bT=\Arr(E,k)$ if $\bT\in P(E,k)$ since $\min\left\{\nu(\bT')\left| \bT'\in P(E,k), \bT<\bT'\right\}\right. =\nu(\bT)+1$.
Moreover, if $\cA\in V_{\bT}$ for some $\bT\in Q(E,k)\setminus P(E,k)$, then $\cA$ is non-very generic.
An example where a $(\bT,r )$-singularity variety is not equal to $V_{\bT}$ can be found in Example \ref{Exam:SeYa}; take $\cA$ and $X$ as in Example \ref{Exam:SeYa}, then $\min\left\{\nu(\bT')\left| \bT'\in P(n,k), \bT<\bT'\right\}\right. =6=\rank X+2$ holds.

\begin{exam}\label{Exam:NVGV}\mbox{}
Consider the same intersection as in Example \ref{Exam:Crapo}.
For $n=6,k=2$, $\bT=\{123,156,246,345\}$ is in $Q(6,2)$ but not in $P(6,2)$.
The element larger than $\bT$ in $P(6,2)$ is only $\{123456\}$, hence $\min\left\{\nu(\bT')\left| \bT'\in P(n,k), \bT< \bT'\right\}\right.=6-2=4$. 
Then 
\begin{align*}
V_{(\bT,r)}=&{\{\cA\in\Arr(6,2)\mid \rank D_{123}(\cA)\cap D_{156}(\cA)\cap D_{246}(\cA)\cap D_{345}(\cA)\leq r\}}
\end{align*}
If $\min\left\{\nu(\bT')\left| \bT'\in P(n,k), \bT<\bT'\right\}\right. =r+1$, then we have
$V_\bT=V_{(\bT,3)}=\{\cA\in\Arr(6,2)\mid \Delta_{16}\Delta_{24}\Delta_{35}-\Delta_{15}\Delta_{26}\Delta_{34}=0\}$ where $\alpha_i$ is a normal vector of $H_i$ and $\Delta_{ij}=\det(\alpha_i\alpha_j)$.

Similarly, for $n=7,k=2$, consider $\bT=\{123,147,156,246,357\}$ in $Q(7,2)$ but not in $P(7,2)$.
Then $\min\{\nu(\bT')\mid \bT'\in P(n,k), \bT< \bT'\}=7-2=5$.
We have $V_\bT=V_{(\bT,4)}$ and the defining equation of $V_{\bT}$ is 
\begin{equation}\label{Eq:quint}
\Delta_{14}\Delta_{16}\Delta_{27}\Delta_{35}-\Delta_{17}\Delta_{15}\Delta_{26}\Delta_{34}=0.
\end{equation}
The author shows this is the defining equation in \cite{SaSe24}.
\end{exam}

We now explain the importance of considering $V_\bT$; that is, the BBA-condition is a necessary and sufficient condition for determining very generic arrangements.

\begin{lemma}\label{lemma:NonVeryGeneric}
Let $\cA\in\gArr(n,k)$.
\begin{enumerate}
\item Let $X$ be a non-very generic intersection in $\cL(\cB(\cA))$. If $X$ satisfies the BBA-condition, then there exists an intersection $Y\in \cL(\cB(\cA))$ which fails the BBA-condition and satisfies $X<Y$.
\item A generic arrangement $\cA$ is non-very generic if and only if there exists a non-very generic translation $\cA^t$.
\end{enumerate}
\end{lemma}
\begin{proof}
\begin{enumerate}
\item If there is no such intersection $Y$, then the length of every maximal chain in $\cL(\cB(\cA))$ containing $X$ is at most $n-k$. This contradicts that the length of a maximal chain of a geometric lattice of rank $n-k$ is $n-k+1$.
\item Suppose that $\cA$ is non-very generic, but every intersection in $\cL(\cB(\cA))$ is very generic.
Then there is a proper subposet $P'$ of $P(n,k)$ with the same grading $\nu$ such that $P'\cong \cL(\cB(\cA))$.
In other words, there exists $\bT_0\in P(n,k)\setminus P'$.
Let $\bT_1$ be the canonical presentation of $\bigcap_{S\in \bT_0}D_S$,
then it holds that $\nu(\bT_1)=\rank \bigcap_{S\in \bT_0}D_S<\nu(\bT_0)$.
Since $\bT_0$ is not a canonical presentation of $X=\bigcap_{S\in \bT_0}D_S$, there exists a canonical presentation $\bT_1$ of $X$ satisfying the BBA-condition with $\bT_1>\bT_0$.
This leads to $\rank X<\nu(\bT_0)<\nu(\bT_1)$.
However, this contradicts $\rank X=\nu(\bT_1)$.
\end{enumerate}
\end{proof}

The above lemma and Proposition \ref{prop:NVGV} yield the following theorem.

\begin{thm}\label{thm:NVGV2}
The set of generic but non-very generic arrangements in $W$ with $n$ hyperplanes is
\begin{equation*}
\bigcup_{\bT\in Q(n,k)\setminus P(n,k), \nu(\bT)=\abs{\bigcup_{S\in \bT}S}-k} V^\circ_{\bT}.
\end{equation*}
\end{thm}
\begin{proof}
    Let $\cA\in\Arr^\circ(n,k)$ be a non-very generic arrangement.
    By Lemma \ref{lemma:NonVeryGeneric}, there exists a non-very generic intersection $X\in\cL(\cB(\cA))$ and $\bT(X)\notin P(n,k)$.
    Hence, we obtain $\cA\in V^\circ_{\bT(X)}$.
    If $\nu(\bT(X))>\abs{\bigcup_{S\in \bT(X)}S}-k$, we can obtain $\bT_0$ such that $\bT_0<\bT(X)$ and $\nu(\bT_0)=\abs{\bigcup_{S\in \bT_0}S}-k$ by replacing some elements $S$ in $\bT$ with smaller elements $S'\subset S$.
    It follows from  Proposition \ref{prop:NVGV} that $\cA\in V^\circ_{\bT(X)}\subset V^\circ_{\bT_0}$.
\end{proof}


\section{Generalized Crapo configuration; wheels}\label{Sec:wheel}
We discuss non-very generic line arrangements, the generalized Crapo configurations, and the corresponding $(\bT,r)$-singularity varieties.
In this section, we assume $n\geq 3$.

We define two families of subsets of $[2n]$ as follows:
\begin{align*}
\bW_{2n}=&\left\{\{2i-1,2i,2i+1\}\subset [2n]\mid i\in [n] \right\}\cup\{\{2,4,\dots,2n\}\},\\
{\bW^\dagger}_{2n}=&\left\{\{2i,2i+1,2i+2\}\subset [2n]\mid i\in [n] \right\}\cup\{\{1,3,\dots,2n-1\}\}.
\end{align*}
Here, we identify $[2n]$ with the cyclic group $\bZ/2n\bZ$.
We call $\bW_{2n}$ a $2n$-\textit{wheel}, and we call ${\bW^\dagger}_{2n}$ its twin.
Furthermore, families $\bW$ and $\bW^\dagger$ of subsets of a finite set $E$ of cardinality $2n$, induced by a bijection $[2n]\to E$, are called a $2n$-wheel and its twin.

A $2n$-wheel arrangement $\cW_{2n}$ is an affine line arrangement $\{H_1,\dots,H_{2n}\}$ with the condition that  $\bigcap_{i\in S}H_i\neq \emptyset$ and $H_j\cap \bigcap_{i\in S}H_i=\emptyset$ for all $j\notin S$ for each $S \in \bW_{2n}$ (cf. Figure \ref{fig:2n-Wheel}).
This name comes from the fact that the intersection graph of $\bW_{2n}$ forms a wheel graph.
It is straightforward to verify that $\bW_{2n}\in Q(2n,2)$ and $\bW_{2n}\notin P(2n,2)$.
Since $\bW_{2n}\setminus\{123\}$ satisfies the BBA-condition, we have
\begin{equation*}
2n-3=\nu(\bW_{2n}\setminus\{123\})<\min\{\nu(\bT)\mid \bT\in P(2n,2),\bW_{2n}\}\leq 2n-2.
\end{equation*}
Hence, we obtain $\min\left\{\nu(\bT')\left| \bT'\in P(n,k), \bW_{2n}< \bT'\right\}\right.-1=4$
and $V_{\bW_{2n}}=V_{(\bW_{2n},2n-3)}$.

\begin{figure}[h]
\centering
\begin{tabular}{cc}
	\begin{minipage}{0.4\linewidth}
	\centering
	\begin{tikzpicture}[scale=1]
		\foreach \nlines in {6}{\foreach \ang in {39}{
		\foreach \n in {2,4,...,12}{
		\draw[rotate=(\n/2-1)*\ang](-0.5,0) --node[pos=1.1] {$H^t_{\n}$}(2.2,0);
		}
		\foreach \m in {1,3,...,11}{
		\draw[rotate=((\m-1)/2)*\ang](2,0)--node[auto=left] {$H^t_{\m}$} (2*cos{-\ang},2*sin{-\ang});
		}
		\draw[rotate=-\ang](-0.5,0) --node[pos=1.1] {$H_{2n}$}(2.2,0);
	    \node[rotate=-\ang] (start) at (2*cos{(-7-\ang)},2*sin{(-7-\ang)}){$\vdots$};
	    \node[rotate=(\nlines-1)*\ang+180] (end) at (2*cos{(7+(\nlines-1)*\ang)},2*sin{(7+(\nlines-1)*\ang)}){$\vdots$};}}
	\end{tikzpicture}
	\caption{$2n$-wheel arrangement}\label{fig:2n-Wheel}
	\end{minipage}&
	\begin{minipage}{0.4\linewidth}
	\centering
	\begin{tikzpicture}[scale=1.2]
		\draw (-0.2,0)--(3.5,0) node[pos=1.1]{$H^t_{2n+2}$};
		\draw (-0.2,-0.1)--(2.4,1.2) node[pos=1.1]{$H^t_{2}$};
		\draw (-0.2,0.1)--(2.4,-1.2) node[pos=1.1]{$H^t_{2n}$};
		\draw (1.8,1.2)--(3.3,-0.3) node[pos=1.1]{$H^t_{1}$};
		\draw (1.8,-1.2)--(3.3,0.3) node[pos=1.1]{$H^t_{2n+1}$};
		\draw (1.6,1.1)--(2.4,0.9) node[pos=1.3]{$H^t_{3}$};
		\draw (1.6,-1.1)--(2.4,-0.9) node[pos=1.5]{$H^t_{2n-1}$};
		\draw (2,-1.3)--(2,1.3) node[pos=1.1]{$H_{\star}$};
		\node[rotate=-30] (start) at (1.5,1.5){$\dots$};
		\node[rotate=30] (end) at (1.5,-1.5){$\dots$};
	\end{tikzpicture}
	\caption{add $H_\star$ to $2n+2$-wheel}\label{fig:thmwheel} 
	\end{minipage}
	\end{tabular}
\end{figure}

\begin{thm}\label{thm:wheel}
Let $\cA=\{H_1,\ldots,H_{2n}\}$ be a generic line arrangement consisting of $2n$ lines in $\bK^2$.
The arrangement $\cA$ admits a translation into the $2n$-wheel arrangement if and only if $\cA$ satisfies the equation:
\begin{equation}\label{n-wheelEq}
\prod_{i=1}^n\Delta_{2i,2i-1}-\prod_{i=1}^n{\Delta_{2i,2i+1}}=0.
\end{equation}
In other words, the defining equation of $V^\circ_{\bW_{2n}}$ is Equation \eqref{n-wheelEq}.
\end{thm}

When $n=3$, this equation appears different from Equation \eqref{Eq:quad}, but we have that these equations are consistent with the three-terms Pl\"ucker relation: $\Delta_{ab}\Delta_{cd}-\Delta_{ac}\Delta_{bd}+\Delta_{ad}\Delta_{bc}=0$.
The equivalence between $\cA$ satisfying Equation \eqref{Eq:quad} and the existence of a non-very generic intersection whose canonical presentation is $\{123,156,246,345\}$ in $\cL(\cB(\cA))$ still holds even if $H_1=H_4$, $H_2=H_5$, or $H_3=H_6$.
It can be proved similarly to when $\cA$ is generic.

\begin{proof}
We prove the theorem by induction on $n$.
The case $n=3$ has already been proven.
Suppose that the claim holds for $n$.

First, suppose a translation $\cA^t$ admitting a $(2n+2)$-wheel arrangement exists.
Then, let $H_\star$ be the line through $H^t_1\cap H^t_2\cap H^t_3$ and $H^t_{2n-1}\cap H^t_{2n}\cap H^t_{2n+1}$.
Let $\alpha_\star$ be the normal vector of $H_\star$.
Now, $\{\{\star,1,2\},\{\star,2n,2n+1\},\{1,2n+1,2n+2\},\{2,2n,2n+2\}\}$ forms a $6$-wheel and $\bW_{2n+2}\setminus\{\{1,2,3\},\{2n-1,2n,2n+1\},\{1,2n+1,2n+2\},\{2,4,\dots,2n+2\}\}\cup\{\{\star,2,3\},\{2n-1,2n,\star\},\{2,4,\dots,2n\}\}$ forms a $2n$-wheel (cf. Figure \ref{fig:thmwheel}).

Then, by assumption,
\begin{equation}\label{n-wheelEq2}
    \frac{\det(\alpha_2\alpha_\star)}{\Delta_{2,3}}\frac{\Delta_{2n,2n-1}}{\det(\alpha_{2n}\alpha_\star)}\prod_{i=2}^{n-1}\frac{\Delta_{2i,2i-1}}{\Delta_{2i,2i+1}}=1,
\end{equation}
and
\begin{equation*}
    {\det(\alpha_2\alpha_\star)\Delta_{2n,2n+1}\Delta_{2n+2,1}}-{\Delta_{2,1}\det(\alpha_{2n}\alpha_\star)\Delta_{2n+2,2n+1}}=0
\end{equation*}
is equivalent to
\begin{equation}\label{n-wheelEq3}
	\frac{\det(\alpha_2\alpha_\star)}{\det(\alpha_{2n}\alpha_\star)}=\frac{\Delta_{2,1}\Delta_{2n+2,2n+1}}{\Delta_{2n,2n+1}\Delta_{2n+2,1}}
\end{equation}
holds. Combining them yields the given equation \eqref{n-wheelEq}.

Conversely, assume that Equation \eqref{n-wheelEq} holds.
Taking $\alpha_\star$ such that Equation \eqref{n-wheelEq3} is satisfied, we see from assumption that Equation \eqref{n-wheelEq2} is also satisfied.
Equation \eqref{n-wheelEq2} provides an arrangement isomorphic to the $2n$-wheel, so removing $H_\star$ gives us the arrangement $\cW_{2n+2}$.
\end{proof}

\begin{coro}
The $(\bW_{2n},2n-3)$-singularity variety and its twin are the same, i.e., $V_{\bW_{2n}}=V_{{\bW^\dagger}_{2n}}$.
\end{coro}

\begin{defi}
We call a family of subsets of $\bT$ a \textit{degenerated wheel} (resp. a \textit{degenerated wheel minus a hub}) if $\bT$ is formed by
$\{\{i_{l},i_{l+1},j_l\}\mid l=1,\dots,n \}\cup\{\{j_1,\ldots,j_n\}\}$ (resp. $\{\{i_{l},i_{l+1},j_l\}\mid l=1,\dots,n \}$) with the conditions;
\begin{itemize}
\item For each $l\in[n]$, $i_l$ does not belong to any element of $\bT$ without $\{i_{l-1},i_{l},j_{l-1}\},\{i_{l},i_{l+1},j_l\}$ (allow $j_l=j_{l'}$ for some $l,l'\in[n]$), and
\item $i_l= i_{l'}$ if and only if $l-l'\in n\bZ$.
\end{itemize}
\end{defi}

\begin{prop}\label{prop:DegenWheel}
Let $\cA$ be a generic line arrangement, and let $X$ be a minimal non-very generic intersection in $\cL(\cB(\cA))$.
If $\bT(X)$ satisfies the following conditions:
\begin{itemize}
\item there exists a subcollection of $\bT(X)$ formed by
$\{\{i_{l},i_{l+1},j_l\}\mid l=1,\dots,n \}$ which is a degenerated wheel minus a hub, and
\item $i_l\in S\in \bT(X)$ implies $S\in \{\{i_{l-1},i_{l},j_{l-1}\},\{i_{l},i_{l+1},j_l\}\}$,
\end{itemize}
then the normal vectors of $\cA$ satisfy the following equation:
\begin{equation*}
\prod_{l=1}^n\Delta_{j_l i_l}-\prod_{l=1}^n{\Delta_{j_l i_{l+1}}}=0.
\end{equation*}
\end{prop}
\begin{proof}
Without loss of generality, we assume $i_1<\dots<i_n<j_l$ for $l\in[n]$.
For each component $D_S$ of $X$ and $\iota\in\{3,\dots,\abs{S}\}$, $\alpha_{S,\iota}$ denotes $\alpha_{p_1p_2p_\iota}$, where $S=\{p_{1}<\dots<p_{\abs{S}}\}$.
The orthogonal space $D^\perp_S\subset \bS(\cA)^\ast$ of $D_S$ is spanned by $\{\alpha_{S,3},\dots,\alpha_{S,\abs{S}}\}$.
Hence the orthogonal space of $X$ is spanned by $\bigcup_{S\in\bT(X)}\{\alpha_{S,3},\dots,\alpha_{S,\abs{S}}\}$ whose cardinality is $\nu(\bT(X))=\sum_{S\in\bT(X)}\rank D_S$.
Since $X$ is non-very generic, $\dim D^\perp_S=\rank X<\nu(\bT(X))$.
Thus $\bigcup_{S\in\bT(X)}\{\alpha_{S,3},\dots,\alpha_{S,\abs{S}}\}$ is linearly dependent.
In particular, it is minimally dependent because $X$ is minimal non-very generic.
In other words, there exists non-zero scalars $c_{S,\iota}$ for $S\in\bT(X),\iota\in\{3,\dots,\abs{S}\}$ 
such that $\sum_{S\in\bT(X)}\sum_{\iota=3}^\abs{S}c_{S,\iota}\alpha_{S,\iota}=0$.
Fix such coefficients $c_{S,\iota}$.
The above equation also holds for each coefficient of $e^\ast_i$, i.e., 
\begin{equation*}
\sum_{S\in\bT(X)}\sum_{\iota=3}^\abs{S}c_{S,\iota}a_{S,\iota,i}e^\ast_i=0,
\end{equation*}
where $\alpha_{S,\iota}=\sum_{i=1}^{\abs{\cA}} a_{S,\iota,i}e^\ast_i$.

When $\abs{S}=3$, we write $c_S$ instead of $c_{S,3}$.
The normal vector of $D_{i_1i_{n}j_n}$ is $\alpha_{i_1i_nj_n}=\Delta_{i_nj_n}e^\ast_{i_1}-\Delta_{i_1j_n}e^\ast_{i_n}+\Delta_{i_1i_n}e^\ast_{j_n}$ and the normal vector of $D_{i_li_{l+1}j_l}$ is $\alpha_{i_li_{l+1}j_l}=\Delta_{i_{l+1}j_l}e^\ast_{i_l}-\Delta_{i_l j_l}e^\ast_{i_{l+1}}+\Delta_{i_li_{l+1}}e^\ast_{j_l}$ for $l=2,\dots,n$.
by assumption, $a_{S,\iota,i_l}\neq 0$ if and only if $S=\{i_{l-1},i_{l},j_{l-1}\},\{i_{l},i_{l+1},j_l\}$ for $l=1,\cdots, n$.
Hence, we get
\begin{align*}
c_{i_{1}i_{2}j_{1}}\Delta_{i_{2}j_{1}}+c_{i_{1}i_{n}j_{n}}\Delta_{i_{n}j_n}&=0,\\
-c_{i_{l-1}i_{l}j_{l-1}}\Delta_{i_{l-1}j_{l-1}}+c_{i_{l}i_{l+1}j_{l}}\Delta_{i_{l+1}j_l}&=0\mbox{ for }l=2,\dots,n-1,\\
-c_{i_{n-1}i_{n}j_{n-1}}\Delta_{i_{n-1}j_{n-1}}-c_{i_{1}i_{n}j_{n}}\Delta_{i_{1}j_n}&=0.
\end{align*}
Combining the above equations, we have
\begin{equation*}
1=\frac{c_{i_{1}i_{2}j_{1}}}{c_{i_{1}i_{n}j_{n}}}\left(\prod_{l=2}^{n-1}\frac{c_{i_{l}i_{l+1}j_{l}}}{c_{i_{l-1}i_{l}j_{l-1}}}\right)\frac{c_{i_{1}i_{n}j_{n}}}{c_{i_{n-1}i_{n}j_{n-1}}}=\frac{\Delta_{i_{n}j_n}}{\Delta_{i_{2}j_{1}}}\left(\prod_{l=2}^{n-1}\frac{\Delta_{i_{l+1}j_l}}{\Delta_{i_{l-1}j_{l-1}}}\right)\frac{\Delta_{i_{n-1}j_{n-1}}}{\Delta_{i_{1}j_n}}=\prod_{i=1}^n\frac{\Delta_{j_l i_l}}{\Delta_{j_l i_{l+1}}}.\qedhere
\end{equation*}
\end{proof}


\section{Deletion, contraction, degeneration}\label{Sec:Degen}
Recall the definition of the discriminantal arrangement in the general case, as given in \cite{Cr84, Cr85, BB97, OxWa19,CFW21,FuWang23}.
Let $\cA=\{H_i=\Ker\alpha_i\mid i=1,\dots,n\}$ be a multiarrangement with normal vectors $\{\alpha_1, \ldots, \alpha_n\}$, allowing the $\alpha_i$ to be parallel.
(Multiarrangements are usually defined as a pair of an arrangement and a weight function, but here, they are defined as a multiset indexed by $[n]$. Thus, we allow $\alpha_i=\alpha_j$ for some $i,j\in[n]$.)
As in the case of a generic arrangement, we define $\bS(\cA)=\bigoplus_{i\in[n]}W/H_i$ and $D_S = \{\cA^t \in \bS(\cA) \mid \bigcap_{i\in S} H^t_i \neq \emptyset\}$.
Let $f_S: V \to \bS(\cA)$ be defined by $f_S(v)=\sum_{i\in S} H_i^{\alpha_i(v)}$ for each $S \subset [n]$.
Then, we obtain $D_S = \Im f_S \bigoplus \bS(\cA\setminus S)$.
Thus, $\codim D_S = \abs{S} -\dim \Im((\alpha_i)_{i\in S}: W \to \bK^S) = \abs{S} - \rank(\bigcap_{i\in S} H_i)$.
Hence, we have that $D_S\neq \bS(\cA)$ if and only if $\{\alpha_i\mid i\in S\}$ is linearly dependent, and $D_S$ is a hyperplane if $ \rank(\bigcap_{i\in S} H_i)=\abs{S}-1$.
A subset $C$ of $[n]$ is a \textit{circuit} if $\{\alpha_i \mid i \in C\}$ is a minimal linearly dependent set respect to inclusion.
We denote by $\cC=\cC(\cA)$ denotes the set of circuits in $\cA$.
Then, $D_C$ is a hyperplane in $\bS(\cA)$ for a circuit $C \in \cC$, and $D_C \neq D_{C'}$ whenever $C'$ in $\cC$ is distinct from $C$.
The discriminantal arrangement associated to $\cA$ is a hyperplane arrangement consisting of all $D_C$ for all circuits $C\in \cC$.
If the cardinality of all circuits is $k+1$, $\cA$ is a generic arrangement.
Thus, this is a generalization of what was defined in Section \ref{Sec:DefOfDisArr}.
Since we are interested in generic arrangements, we will focus on the case where the cardinality of each circuit is either $2$ or $k+1$.
In other words, they are generic arrangements with some additional parallel hyperplanes.
If $C=\{i,j\}$, then $\alpha_C=c_ie^\ast_i-c_je^\ast_j$, where $c_i,c_j$ satisfy $c_i\alpha_i=c_j\alpha_j$.
Here $e_i$'s, in which each $e_i$ spans $W/H_i$ for each $i\in[n]$, are the canonical basis of $\bS(\cA)$, and $e^\ast_i$ are the dual basis.
\begin{exam}
	Let the coefficient field be $\bR$. Take the vectors
	\begin{equation*}
	(\alpha_1\alpha_2\alpha_3\alpha_4)=\begin{pmatrix}1&2&0&1\\0&0&1&1\end{pmatrix}.
	\end{equation*}
	Then, the circuits of the arrangement $\cA$ whose normal vectors are $\alpha_1,\alpha_2,\alpha_3$, and $\alpha_4$ are $\cC(\cA)=\{12,134,234\}$.
	Hence, $\cB(\cA)$ consists of three hyperplanes in a $4$-dimensional space $\bS(\cA)$, and the normal vectors are
	\begin{equation*}
		\alpha_{12}=-2e^\ast_1+e^\ast_2,\quad
		\alpha_{134}=-e^\ast_1-e^\ast_3+e^\ast_4,\quad
		\alpha_{234}=-e^\ast_2-2e^\ast_3+2e^\ast_4.
	\end{equation*}
\end{exam}

As is well known, deletion and restriction are fundamental operations in the study of arrangements.
For a (multi)arrangement $\cA$ in $W$ and $H_i\in\cA$, the deletion is defined as $\cA\setminus H_i=\cA\setminus\{H_i\}$ and the restriction as $\cA^{H_i}=\{H_i\cap H\mid H\in\cA\setminus H_i\}$.
Then, $\cC(\cA\setminus H_i)$ consists of all circuits of $\cA$ not containing $i$, i.e., $\{C\mid i\notin C\in \cC(\cA)\}$.
Furthermore, $\cC(\cA^{H_i})$ consists of all the minimal elements of $\{C\setminus i\mid i\in C\in \cC(\cA)\}$.
The discriminantal arrangement associated to restriction $\cA^{H_i}$ was discussed in \cite{LiSe18}, which we will revisit here.

\begin{prop}\label{prop:DelCont}
Let $\cA$ be an arrangement and let $H_i\in \cA$.
\begin{enumerate}[i)]
\item The discriminantal arrangement $\cB(\cA\setminus H_i)$ associated to $\cA\setminus H_i$ is linearly equivalent to $\{D_C\cap\bS(\cA\setminus H_i) \mid i\notin C\in \cC(\cA)\}.$
In particular, if $H_i=H_j$ for some $j\neq i$, then $\cB(\cA\setminus H_i)$ is linearly equivalent to $\cB(\cA)^{D_{ij}}$.
\item The discriminantal arrangement $\cB(\cA^{H_i})$ associated to $\cA^{H_i}$ is linearly equivalent to 
\begin{align*}
&\{D_C\cap\bS(\cA\setminus H_i) \mid i\in C\in \cC(\cA)\}\\
&\cup\{D_{C\cup\{i\}}\cap\bS(\cA\setminus H_i) \mid i\notin C\in \cC(\cA), C\cup\{i\}\mbox{ contains no circuits of }\cA\mbox{ except }C\}.
\end{align*}
In particular, if $\cA$ is a generic arrangement, then $\cB(\cA^{H_i})\cong\{D_C\cap\bS(\cA\setminus H_i) \mid i\in C\in \cC(\cA)\}$.
\end{enumerate}
\end{prop}
\begin{proof}
\begin{enumerate}[i)]
\item The first claim follows from $\cC(\cA\setminus H_i)=\{C\in \cC(\cA)\mid i\notin C\}$.
If $\cA$ has $j\in [n]$ such that $H_i=H_j$, then $D_{ij}=\{\cA^t\mid H^t_i=H^t_j\}$.
Hence the isomorphism between $D_{ij}$ and $\bS(\cA\setminus{H_i})$ given by $\cA^t\mapsto \cA^t\setminus H_i^t$ provides the linear equivalence $\cB(\cA)^{D_{ij}}\cong \cB(\cA\setminus H_i)$.

\item Fix a natural isomorphism between $\bS(\cA^{H_i})=\bigoplus_{H\in \cA\setminus H_i}H_i/(H_i\cap H) $ and $\bS(\cA\setminus{H_i})=\bigoplus_{H\in \cA\setminus H_i}W/H$.
Let $(\cA\setminus H_i)^t\in D_S(\cA\setminus H_i)$; then $\bigcap_{i\in S}(H_i^t\cap H_i)=H_i\cap \bigcap_{i\in S}H_i^t\neq \emptyset.$
Thus if $i\in C\in \cC(\cA)$, then $C\setminus i$ is a circuit of $\cA^{H_i}$ and $D_{C\setminus i}(\cA^{ H_i})= D_C(\cA)\cap\bS(\cA\setminus H_i)$.
Assume $i\notin C\in \cC(\cA)$: then $D_{C}(\cA^{ H_i})\cong D_{C\cup \{i\}}(\cA)\cap\bS(\cA\setminus H_i)$.
If $C\cup\{i\}$ contains circuits $C'$ of $\cA$ other than $C$, then $D_{C\cup \{i\}}\subsetneq D_{C'}$.
In this case, $D_{C\cup \{i\}}$ is not a hyperplane in $\bS(\cA)$.
Conversely, if $C\cup\{i\}$ contains no circuits except $C$, then $D_{C\cup\{i\}}$ is a hyperplane in $\bS(\cA)$.
\end{enumerate}
\end{proof}

Therefore, the very generic arrangements are closed under the operations deletion and restriction.
\begin{coro}
    Let $\cA$ be a generic arrangement.
    If $\cA\setminus H$ or $\cA^H$ is a non-very generic arrangement for some hyperplane $H\in\cA$, then $\cA$ is also non-very generic.
\end{coro}

In particular, we call the operation of deleting a parallel element a \textit{degeneration}.
\begin{defi}
Let $i\in[n-1]$, and let $\bT\in Q(n,k)$.
Define 
\begin{equation*}
\tilde{\bT}_{n\to i}=\{S\cup\{i\}\setminus\{n\}\mid n\in S\in\bT\}\cup\{S\mid n\notin S\in\bT\}\setminus\binom{[n]}{k}
\end{equation*}
and $\bT_{n\to i}$ to be a minimal element of $Q(n-1,k)$ equal to, or larger than, $\tilde{\bT}_{n\to i}$.
The family $\bT_{n\to i}$ of subsets of $[n-1]$ is said to be a \textit{degeneration} of $\bT$ from $n$ to $i$.
We write $\gamma=\gamma(\bT;n,i)$ as the cardinality of ${\{S\in \bT\mid \{n,i\}\subset S\}}$.
\end{defi}

\begin{exam}
	\begin{figure}[h]
	\centering
	\begin{tabular}{cc}
	\begin{minipage}{0.4\linewidth}
	\centering
	\begin{tikzpicture}[scale=2.5]
		\draw (0,-0.2)--(0,1.2) node[pos=-0.1]{$H^t_{1}$};
		\draw (0.6,-0.2)--(-0.1,1.2) node[pos=-0.1]{$H^t_{2}$};
		\draw (-0.2,1.2)--(1.2,-0.2) node[pos=-0.1]{$H^t_{3}$};
		\draw (-0.2,0.6)--(1.2,-0.1)node[pos=-0.1]{$H^t_{4}$};
		\draw (-0.2,0)--(1.2,0) node[pos=-0.1]{$H^t_{5}$};
		\draw (-0.2,-0.2)--(0.8,0.8) node[pos=1.1]{$H^t_{6}$};
	\end{tikzpicture}
	\end{minipage}&
	\begin{minipage}{0.4\linewidth}
	\centering
	\begin{tikzpicture}[scale=2.5]
		\draw (0,-0.2)--(0,1.2) node[pos=-0.1]{$H^t_{1}$};
		\draw (0.6,-0.2)--(-0.1,1.2) node[pos=-0.1]{$H^t_{2}$};
		\draw (-0.2,1.2)--(1.2,-0.2) node[pos=-0.1]{$H^t_{3}$};
		\draw (-0.2,0.76)node[at={(-0.35,0.85)}]{$H^t_{4}$}--(1,-0.2);
		\draw (-0.2,0)--(1.2,0) node[pos=-0.1]{$H^t_{5}$};
		\draw (-0.2,-0.2)--(0.8,0.8) node[pos=1.1]{$H^t_{6}$};
		\draw (-0.2,0.72)node[at={(-0.35,0.65)}]{$H^t_{7}$}--(1.2,-0.12);
	\end{tikzpicture}
	\end{minipage}
	\end{tabular}
		\caption{representatives of the minimal non-very generic intersection with $6$ or $7$ lines}\label{fig;degen} 
	\end{figure}

Let $\bT$ be $\{123,147,156,246,357\}$.
The degeneration of $\bT$ from $7$ to $4$ is $\bT_{7\to 4}=\{123,156,246,345\}$.
Refer to the two $(\bT,r )$-singularity varieties in Example \ref{Exam:NVGV}; these are defined by equation \eqref{Eq:quad} and \eqref{Eq:quint}, i.e.
\begin{align*}
V_\bT=&\{\cA\in\Arr(7,2)\mid \Delta_{14}\Delta_{16}\Delta_{27}\Delta_{35}-\Delta_{17}\Delta_{15}\Delta_{26}\Delta_{34}=0\}\\
V_{\bT_{7\to4}}=&\{\cA\in\Arr(6,2)\mid \Delta_{16}\Delta_{24}\Delta_{35}-\Delta_{15}\Delta_{26}\Delta_{34}=0\}
\end{align*}
 where $\alpha_i$ is a normal vector of $H_i$, and $\Delta_{ij}=\det(\alpha_i\alpha_j)$.
For real or complex coefficients, the degeneration of $\bT$ from $4$ to $7$ corresponds to continuously deforming the representatives and taking the limit $H^t_7\to H^t_4$ while maintaining the canonical presentation (see Figure \ref{fig;degen}).
Equation \eqref{Eq:quad} is obtained by substituting $\alpha_4$ for $\alpha_7$ in Equation \eqref{Eq:quint}.
It says that if $\cA=\{H_1,\dots, H_7\}\in V_{\bT}$ and $H_4=H_7$, then $\cA\setminus H_7\in V_{\bT_{7\to4}}$.
\end{exam}
Generalizing the above example, we obtain the following theorem.

\begin{thm}\label{thm:degen}
Let $\bT\in Q(n,k)$, and let $\bT_{n\to j}$ be the degeneration of $\bT$ from $n$ to $j$.
\begin{enumerate}
\item If $\abs{\bigcup_{S\in\bT_{n\to j}}S}=\abs{\bigcup_{S\in\bT}S}-1$, $\bT\notin P(n,k)$, $\gamma(\bT;n,j)=1$, and $\tilde\bT_{n\to j}\in Q(n-1,k)$, then we have $\bT_{n\to j}\notin P(n-1,k)$.
\item If $\gamma(\bT;n,j)>0$, $\cA=\{H_1,\dots,H_n\}\in V_{(\bT,r)}$, and $H_n=H_j$, then we have $\cA\setminus H_n\in V_{(\bT_{n\to j},r-1)}$.
\end{enumerate}
\end{thm}
\begin{proof}
\begin{enumerate}
\item Since $\bT\notin P(n,k)$, there exists $\bT'\subset\bT$ such that $\nu(\bT')\geq \abs{\bigcup_{S\in\bT'}S}-k$.
Let us show that $\bT'_{n\to j}\notin P(n-1,k)$.
by assumption, $\tilde\bT_{n\to j}=\bT_{n\to j}$ for each $\bT'\subset \bT$.
We obtain $\nu(\bT'_{n\to j})=\nu(\tilde\bT'_{n\to j})=\nu(\bT')-\gamma(\bT';n,j)$.
Because $\gamma(\bT';n,j)\leq\gamma(\bT;n,j)$ for $\bT'\subset \bT$, we get $0\leq \gamma(\bT';n,j)\leq 1$.
If $\gamma(\bT';n,j)=1$, then $\abs{\bigcup_{S\in\bT'_{n\to j}}S}=\abs{\bigcup_{S\in\bT'}S\setminus \{n\}}=\abs{\bigcup_{S\in\bT'} S}-1$, hence we have $\nu(\bT'_{n\to j})\geq \abs{\bigcup_{S\in\bT'_{n\to j}}S}-k$.
If $\gamma(\bT';n,j)=0$, then we have $\nu(\bT'_{n\to j})=\nu(\bT')\geq \abs{\bigcup_{S\in\bT'}S}-k\geq \abs{\bigcup_{S\in\bT'_{n\to j}}S}-k$.

\item 
We identify $\cB(\cA\setminus H_n)$ with $\cB(\cA)^{D_{nj}}$ via Proposition \ref{prop:DelCont}.
by assumption, $D_{nj}\supset \bigcap_{S\in \bT}D_S$ and $\rank_{\cB(\cA)} \bigcap_{S\in \bT}D_S\leq r$.
Hence
\begin{equation*}
\bigcap_{S\in \bT}D_S=D_{nj}\cap\bigcap_{S\in \bT}D_S=D_{nj}\cap\bigcap_{S\in \bT_{n\to j}}D_S.
\end{equation*}
and 
\begin{align*}
\rank_{\cB(\cA)^{D_{nj}}} D_{nj}\cap\bigcap_{S\in \bT_{n\to j}}D_S=&\rank_{\cB(\cA)} D_{nj}\cap\bigcap_{S\in \bT_{n\to j}}D_S-1\\
=&\rank_{\cB(\cA)} \bigcap_{S\in \bT}D_S-1\leq r-1.\qedhere
\end{align*}
\end{enumerate}
\end{proof}


\section{Applications}\label{Sec:ExamNVVar}
In this section, we provide some applications of the degeneration construction and the wheel arrangements.
The $2n$-wheel is defined for a generic arrangement consisting of $2n$ hyperplanes.
However, it can be extended to a general arrangement by adding the following assumptions:  $H_{i}\neq H_{i+1}$ and $H_{i}\neq H_{i+2}$ for each $i\in[2n]$.
Furthermore, under this assumption, Theorem \ref{thm:wheel} still holds; that is, the $(\bT,r )$-singularity variety $V_{\bW_{2n}}$ is defined by Equation \eqref{n-wheelEq}.
We define $\bL_{2n+2}=\{\{2i-1,2i,2n+1\}\mid i\in[n] \}\cup\{\{2i,2i+1,2n+2\}\mid i\in[n-1]\}\cup\{\{1,2n,2n+2\}\}$; $\bL_{2n+2}$ is called a $2n+2$-\textit{ladder}.
Then, $\bL_{2n+2}\in Q(2n+2,2)\setminus P(2n+2,2)$.
\begin{prop}\label{2n+2ladder}
The defining equation of $V_{\bL_{2n+2}}$ is 
\begin{equation*}
\prod_{i=1}^n\Delta_{2i,2n+2}\Delta_{2i-1,2n+1}-\prod_{i=1}^n{\Delta_{2i,2n+1}\Delta_{2i-1,2n+2}}=0.
\end{equation*}
\end{prop}
\begin{proof}
Because $\bL_{2n+2}$ can be obtained by a degeneration of $\bW_{4n}$ from $4i-2(i=2,\dots,n)$ to $2$ and from $4i(i=2,\dots,n)$ to $4$ and by permutations of indices.
Thus, the equation follows from Theorem \ref{thm:degen} and Theorem \ref{thm:wheel}.
\end{proof}

Moreover, by applying Theorem \ref{thm:degen} to Theorem \ref{thm:wheel} in the same way as the proof above and combining it with Proposition \ref{prop:DegenWheel}, we obtain the following proposition.
\begin{prop}\label{coro:wheel}
Let $\cA$ be a generic line arrangement and let $X$ be a minimal non-very generic intersection in $\cL(\cB(\cA))$.
If there exists a subcollection of $\bT(X)$ formed by
$\{\{i_{l},i_{l+1},j_l\}\mid l=1,\dots,n \}$ which is a degenerated wheel minus a hub,
then there exists a non-very generic intersection $Y$ in $\cL(\cB(\cA))$, whose canonical presentation is a degenerated wheel $\{\{i_{l},i_{l+1},j_l\}\mid l=1,\dots,n \}\cup\{\{j_1,\ldots,j_n\}\}$.
\end{prop}

Let us list the $(\bT,r )$-singularity varieties corresponding to minimal non-very generic intersections in $\Arr(8,2)$.
Assume that a line arrangement $\cA$ with $8$ lines has a minimal non-very generic intersection $X$ in $\cL(\cB(\cA))$.
We will denote $n'$ by $\abs{\bigcup_{S\in \bT(X)}S}$.
Because $X$ is minimal non-very generic, $\rank X+1=\nu(\bT(X))=\sum_{S\in\bT(X)}(\abs{S}-2)=\sum_{S\in\bT(X)}\abs{S}-2\abs{\bT(X)}$ and each line in $\bigcup_{S\in \bT(X)}S$ is contained in at least two subsets $S\in\bT(X)$.
Hence, $2n'\leq \sum_{S\in\bT(X)}\abs{S}$.
After subtracting $2\nu(\bT(X))\geq 2\abs{\bT(X)}$ from both sides of this inequality, 
$2\left(n'-\nu(\bT(X))\right)\leq2\left(n'-\abs{\bT(X)}\right)\leq \nu(\bT(X))$, then we get $2n'/3\leq \nu(\bT(X))=\rank (X)+1$.
Since $n'\geq 6$, we should consider the cases $n'=6,7,8$.
In the case $n'=6,7$, since $2n'/3-1\leq \rank X\leq n'-3$, $X$ is either as in Example \ref{Exam:NVGV} or permutations of them.
In the case $n'=8$, since $13/3\leq \rank X\leq 5$, we obtain $\rank X=5$ and $\nu(\bT(X))=6$.
If $\abs{S}=3$ for all $S\in \bT(X)$, then there exist two lines, $H_{i_1}$ and $H_{i_2}$, in $\cA$ such that $\abs{\{S\in\bT(X)\mid i_l\in S\}}=3$ for $l=1,2$. Then, $\bT(X)$ is an $8$-ladder or a degenerated $10$-wheel.
If $\abs{S}=3$ for some $S\in \bT(X)$, then $\bT(X)$ is a permutation of the $8$-wheel.
Then, the equations follow from Theorem \ref{thm:degen} and Theorem \ref{thm:wheel}.
Hence, the following proposition holds.

\begin{thm}\label{thm:8lines}
Let $\cA$ be a generic arrangement with $8$ lines in a plane.
The canonical presentation of minimal non-very generic intersections in $\cL(\cB(\cA))$ is one of the following five families or their permutations:
the $6$-wheel $\bW_6$, the degenerated $8$-wheel $(\bW_8)_{8\to 4}$, the $8$-wheel $\bW_8$, the degenerated $10$-wheel with indices $8$ and $9$ swapped, $(((89).\bW_{10})_{10\to 6})_{9\to 4}$.
The defining equations of these $(\bT,r )$-singularity varieties are as follows:
\begin{align*}
\bW_6:&\Delta_{21}\Delta_{43}\Delta_{65}-\Delta_{23}\Delta_{45}\Delta_{61}=0,\\
(\bW_8)_{8\to 4}:&\Delta_{21}\Delta_{43}\Delta_{65}\Delta_{47}-\Delta_{23}\Delta_{45}\Delta_{67}\Delta_{41}=0,\\
\bW_8:&\Delta_{21}\Delta_{43}\Delta_{65}\Delta_{87}-\Delta_{23}\Delta_{45}\Delta_{67}\Delta_{81}=0,\\
\bL_8:&\Delta_{17}\Delta_{37}\Delta_{57}\Delta_{28}\Delta_{48}\Delta_{68}-\Delta_{18}\Delta_{38}\Delta_{58}\Delta_{27}\Delta_{47}\Delta_{67}=0,\\
(((89).\bW_{10})_{10\to 6})_{8\to 4}:&\Delta_{21}\Delta_{43}\Delta_{65}\Delta_{47}\Delta_{68}-\Delta_{23}\Delta_{45}\Delta_{67}\Delta_{48}\Delta_{61}=0.
\end{align*}
\end{thm}

What happens if there are nine or more lines?
Since $2n'/3-1\leq \rank X\leq n'-2$, the rank of $X$ may not be uniquely determined by $n'$.
For example, let $n'=9$, then $\rank X$ could be either $5$ or $6$.
We give an example where $\rank X=5$.

\begin{exam}\label{exam:K33}
Let $\cA$ be a line arrangement with normal vectors (Fig~\ref{fig:9line})
\begin{equation*}
(\alpha_1\dots\alpha_9)=\begin{pmatrix}-4&-3&-2&-1&0&1&2&3&4\\ 1&1&1&1&1&1&1&1&1\end{pmatrix}.
\end{equation*}
Consider $\bT_0=\{123,456,789,147,258,369\}$, and the intersection $X=\bigcap_{S\in \bT_0} D_S$ in $\cL(\cB(\cA))$.
\begin{figure}[t]
\centering
	\begin{tikzpicture}[yscale=0.8, xscale=3]
		\draw (0.4,-3.8)--(1.9,2.2) node[pos=1.05]{$H^t_{1}$};
		\draw (0.4,-3.6)--(1.8,0.6) node[pos=1.05]{$H^t_{2}$};
		\draw (0.4,-3.4)--(1.5,-1.2) node[pos=1.1]{$H^t_{3}$};
		\draw (-0.4,-0.4)--(2.0,2.0) node[pos=-0.08]{$H^t_{4}$};
		\draw (-0.4,0)--(1.8,0) node[pos=-0.09]{$H^t_{5}$};
		\draw (-0.4,0.4)--(1.6,-1.6) node[pos=-0.1]{$H^t_{6}$};
		\draw (-0.8,7)--(1.9,1.6) node[pos=1.05]{$H^t_{7}$};
		\draw (-0.8,7.2)--(1.9,-0.9) node[pos=1.05]{$H^t_{8}$};
		\draw (-0.8,7.3)--(1.5,-1.8) node[pos=1.05]{$H^t_{9}$};
	\end{tikzpicture}\par
\caption{a representative of $X$ in Example \ref{exam:K33}}\label{fig:9line}
\end{figure}
We obtain normals
\begin{equation*}
\begin{pmatrix}
\alpha_{123}\\ \alpha_{456}\\ \alpha_{789}\\ \alpha_{147}\\ \alpha_{258}\\ \alpha_{369}
\end{pmatrix}=
\begin{pmatrix}
-1&2&-1&0&0&0&0&0&0\\
0&0&0&-1&2&-1&0&0&0\\
0&0&0&0&0&0&-1&2&-1\\
-3&0&0&6&0&0&-3&0&0\\
0&-3&0&0&6&0&0&-3&0\\
0&0&-3&0&0&6&0&0&-3
\end{pmatrix}.
\end{equation*}
The non-very generic intersection $X$ of rank $5$ is minimal because $\nu(\bT_0\setminus i)=\bigcap_{C\in \bT_0\setminus i}D_C=4$ where $\bT_0\setminus i=\{C\in \bT_0\mid i\notin C\}$ for $i=1,\dots,9$.
We have $\min\left\{\nu(\bT')\left| \bT'\in P(n,k), \bT< \bT'\right\}\right.=\rank X+2$.
It is easy to check that $\bT_0$ satisfies the BBA-condition.
On the other hand, by Proposition \ref{coro:wheel}, the arrangement $\cA$ belongs to $(\bT,r )$-singularity varieties of nine $8$-wheels  $\bT_1,\ldots ,\bT_9$ and six degenerated $12$-wheels $\bT_{10} ,\ldots,\bT_{15}$ as follows:
\begin{align*}
&\begin{aligned}
\bT_1=&\{123,456,147,258,3678\},&
\bT_2=&\{123,789,147,258,3459\},\\
\bT_3=&\{123,456,147,369,2579\},&
\bT_4=&\{123,789,147,369,2468\},\\
\bT_5=&\{123,456,258,369,1489\},&
\bT_6=&\{123,789,258,369,1567\},\\
\bT_7=&\{456,789,147,258,1269\},&
\bT_8=&\{456,789,147,369,1358\},\\
\bT_9=&\{456,789,258,369,2347\},
\end{aligned}\\
&\begin{aligned}
\bT_{10}=&\{123,456,789,147,258,369,159\},&
\bT_{11}=&\{123,456,789,147,258,369,168\},\\
\bT_{12}=&\{123,456,789,147,258,369,249\},&
\bT_{13}=&\{123,456,789,147,258,369,267\},\\
\bT_{14}=&\{123,456,789,147,258,369,348\},&
\bT_{15}=&\{123,456,789,147,258,369,357\}.
\end{aligned}
\end{align*}
Then we obtain $V_{\bT_0}\subset V_{\bT_i}$ for $i\in[15]$.
Note that, $V_{\bT_1},\dots,V_{\bT_{15}}$ do not intersect transversally.
Clearly, no intersection is larger than $X$ whose canonical presentation is the $8$-wheel.
However, if we consider $X\cap\bigcap_{S\in\bT_1}D_S$, then its rank is $6$ and the canonical presentation is $\{123,456,147,258,36789\}$.
Furthermore, the intersections corresponding to these degenerated $12$-wheels are the first example of intersections that are not minimal non-very generic intersections, but all proper subsets of canonical presentation that satisfy the BBA-condition.
\end{exam}

\section*{Acknowledgment}
This work was supported by JSPS KAKENHI Grant Number JP23KJ0031.
The author would like to thank Yasuhide Numata, Mutsumi Saito, Simona Settepanella, So Yamagata, and anonymous referees for their helpful comments.

\bibliographystyle{alpha}
\bibliography{reference}
\end{document}